\documentclass[english]{amsart}

\usepackage{esint}
\usepackage[svgnames]{xcolor} 
\usepackage[colorlinks,citecolor=red,pagebackref,
hypertexnames=false,breaklinks]{hyperref}
\usepackage{pgf,tikz}
\usepackage{pdfsync}

\usepackage{dsfont}
\usepackage{url}
\usepackage[utf8]{inputenc}
\usepackage[T1]{fontenc}
\usepackage{lmodern}
\usepackage{babel}
\usepackage{mathtools}  
\usepackage{amssymb}
\usepackage{lipsum}
\usepackage{mathrsfs}
\usepackage{color}

\usepackage[skip=2.2pt plus 1pt, indent=12pt]{parskip}

\newtheorem{theorem}{Theorem}[section]

\newtheorem{lemma}{Lemma}[section]

\newtheorem{remark}{Remark}[section]

\numberwithin{equation}{section}

\title[Quantitative uniqueness of continuation]{Comments on the quantitative uniqueness of continuation for evolution equations}

\author[Mourad Choulli]{Mourad Choulli}
\address{Universit\'e de Lorraine}
\email{mourad.choulli@univ-lorraine.fr}

\begin{document}

\begin{abstract}
We establish near-optimal quantitative uniqueness of continuation for solutions of evolution equations vanishing on the lateral boundary. These results were obtained simply by combining existing observability inequalities and energy estimates.
\end{abstract}



\maketitle


\section{Introduction}\label{section1}

Throughout this text, $\Omega$ is a bounded Lipschitz domain of $\mathbb{R}^n$, $n\ge 2$, with boundary $\Gamma$. The unit normal exterior vector field on $\Gamma$ is denoted by $\nu$.  Let $\Omega_\ast$ be an open domain of $\mathbb{R}^n$ fixed arbitrarily so that $\Omega \Subset \Omega_\ast$. Denote by  $\mathbf{g}=(g_{k\ell})$ a symmetric matrix satisfying $g_{k\ell}\in C^{2,1}(\overline{\Omega}_\ast)$, $1\le k,\ell\le n$, and there exists a constant $\varkappa\ge 1$ so that
\[
\varkappa ^{-1}|\xi| ^2\le (\mathbf{g}(x)\xi|\xi) \le \varkappa |\xi|^2,\quad x\in \Omega_\ast ,\; \xi\in \mathbb{R}^n.
\]
Here and henceforth, $(\cdot |\cdot)$ denotes the Euclidian scalar product of $\mathbb{R}^n$.

Let $\Delta_{\mathbf{g}}:=\mathrm{div}(\mathbf{g}\nabla \cdot)$ and consider the usual evolution operators:
\begin{align*}
&H=\Delta _\mathbf{g}-\partial_t\quad \; \mbox{(heat operator)},
\\
&S=\Delta _\mathbf{g}+i\partial_t\quad \mbox{(Schr\"odinger operator)},
\\
&W=\Delta _\mathbf{g}-\partial_t^2\quad \mbox{(wave operator)}.
\end{align*}

Let us recall some classical uniqueness of continuation results for the operators $P$, $S$ and $W$. Let $\omega\Subset \Omega$, $\Gamma_0$ a nonempty open subset $\Gamma$, and set $Q_\omega =\omega \times (0,\mathfrak{t})$ and $\Sigma_0=\Gamma_0\times (0,\mathfrak{t})$ with $\mathfrak{t}>0$. From \cite[Theorem 5.3]{Ch} (resp. \cite[Corollary 5.1]{Ch}) if $u\in H^1((0,\mathfrak{t}),H^2(\Omega))$ satisfies $Hu=0$ and $u_{|Q_\omega}=0$ (interior data) (resp. $u_{|\Sigma_0}=\partial_\nu u_{|\Sigma_0}=0$) (Cauchy data on a part of the lateral boundary) then $u=0$. A quantitative uniqueness of continuation from a Cauchy data for $H$ with a modulus of continuity of logarithmic type was obtained by L. Bourgeois in \cite{Bo} when $\mathbf{g}$ is the identity matrix of $\mathbb{R}^n$, denoted hereinafter by $\mathbf{I}$,  $\Omega=O\setminus D$ is $C^2$ and $\Gamma_0$ is either $\partial O$ or $\partial D$. The general case was proved by the author and M. Yamamoto in \cite{CY} with a modulus of continuity of multiple logarithmic type. 

It follows from \cite[Theorem 6.4]{Ch} (resp. \cite[Corollary 6.1]{Ch}) that there exists a neighborhood $\mathcal{N}$ of $\mathbf{I}$ in $C^{2,1}(\overline{\Omega})$ so that for any $\mathbf{g}\in \mathcal{N}$ and $u\in L^2((0,\mathfrak{t}),H^2(\Omega))\cap H^1((0,\mathfrak{t}),L^2(\Omega))$ satisfying $Su=0$ and $u_{|Q_\omega}=0$ (resp. $u_{|\Sigma_0}=\partial_\nu u_{|\Sigma_0}=0$) then $u=0$ in $\Omega \times (\mathfrak{t}/4, 3\mathfrak{t}/4)$.

 For the wave equation, we know that there exists a constant $\mathfrak{t}_0>0$ only depending on $\Omega$, $\omega$, $\varkappa$ so that for any $\mathfrak{t}>\mathfrak{t}_0$ and $u\in H^2(\Omega\times (0,\mathfrak{t}))$ satisfying $Wu=0$ and $u_{|Q_\omega}=0$ then $u=0$ in $\Omega \times (\mathfrak{t}_0,\mathfrak{t}-\mathfrak{t}_0)$. This result is due to L. Robbiano \cite{Ro} when $\mathbf{g}\in C^1(\overline{\Omega},\mathbb{R}^{n\times n})$. We obtained in \cite{BC} a partial quantitative version of this uniqueness of continuation result with a modulus of continuity of multiple logarithmic type. The general case remains an open problem.
 
We show in this short note that the existing result can be improved and completed when we restrict ourselves to solution vanishing on the lateral boundary. The results we give are obtained simply by combining known observability inequalities and elementary energy estimates.

Before stating our main results, we introduce some definitions and notations. We first recall the pseudo-convexity condition we introduced in \cite{Ch}. To this end, we define
\[
\Lambda_{k\ell}^m(\mathbf{g})(x)=-\sum_{p=1}^n\partial_pg_{k\ell}(x)g_{pm}(x)+2\sum_{p=1}^ng_{kp}(x)\partial_pg_{\ell m}(x),\quad x\in \overline{\Omega},\; 1\le k,\ell,m\le n.
\]
For all $h\in C^1(\overline{\Omega})$, we consider the matrix $\Upsilon_{\mathbf{g}} (h)$ given by
\[
(\Upsilon_{\mathbf{g}} (h))_{k\ell}(x)=\sum_{m=1}^n\Lambda_{k\ell}^m(\mathbf{g})(x)\partial_m h (x),\quad x\in \overline{\Omega} 
\]
and let
\[
\mathfrak{m}_h:= \min_{\overline{\Omega}}|\nabla h|.
\]

We say that $h\in C^2(\overline{\Omega})$ is $\mathbf{g}$-pseudo-convex with parameter $\kappa>0$  if $\mathfrak{m}_h>0$ and 
\[
(\Theta_{\mathbf{g}}(h)(x)\xi|\xi)\ge \kappa |\xi|^2,\quad x\in \overline{\Omega},\; \xi \in \mathbb{R}^n.
\]
Here
\[
\Theta_{\mathbf{g}}(h)=2\mathbf{g}\nabla^2h\mathbf{g}+\Upsilon_\mathbf{g}(h),
\]
where $\nabla^2 h=(\partial_{k\ell}^2h)$.

Noting that $\Theta_{\mathbf{I}}(h)=2\nabla^2h$, we see that $h$ is $\mathbf{I}$-pseudo-convex with parameter $\kappa$  iff $\mathfrak{m}_h> 0$ and $h$ is strongly convex with parameter $\kappa/2$. We  refer to \cite[Example 2.1]{Ch} for examples in the case where $h(x)=|x-x_0|^2$, for some $x_0\in \mathbb{R}^n\setminus \overline{\Omega}$ fixed arbitrarily. 

The following notations will be used hereinafter
\[
Q_t= \Omega \times (0,t),\quad \Sigma_t=\Gamma\times (0,t),\quad t>0,
\]
$\mathfrak{t}>0$ is fixed and $Q=Q_{\mathfrak{t}}$. Define
\begin{align*}
&\mathcal{H}_0(Q)=L^2((0,\mathfrak{t}),H^2(\Omega)\cap H_0^1(\Omega))\cap H^1((0,\mathfrak{t}),H^1(\Omega)),
\\
&\mathcal{H}_1(Q)=L^2((0,\mathfrak{t}),H^2(\Omega)\cap H_0^1(\Omega))\cap H^2((0,\mathfrak{t}),H^1(\Omega)).
\end{align*}
It is worth remarking that, according to Lions-Magenes trace theorems \cite[Chapter 4, Section 2]{LM2}, $u(\cdot,0)\in H_0^1(\Omega)$ whenever $u\in \mathcal{H}_0(Q)$.

We endow $H_0^1(\Omega)$ with the norm $\|\nabla \cdot \|_{L^2(\Omega)}$ and, for any $\theta >1$, $\mathbf{e}_\theta$ will denote the constant of the interpolation inequality (e.g. \cite[Theorem 9.6]{LM1})
\begin{equation}\label{i}
\|w\|_{H_0^1(\Omega)}\le \mathbf{e}_\theta \|w\|_{H^\theta(\Omega)}^{1/\theta}\|w\|_{L^2(\Omega)}^{1-1/\theta},\quad w\in H^\theta (\Omega)\cap H_0^1(\Omega).
\end{equation}

For all $\psi \in C^1(\overline{\Omega})$, we set
\[
\Gamma^\psi =\{x\in \Gamma;\; \partial_{\nu_{\mathbf{g}}}\psi >0\},\quad \Sigma_t^\psi=\Gamma^\psi \times (0,t),
\]
where $\partial_{\nu_{\mathbf{g}}}= (\mathbf{g}\nabla \cdot |\nu)$.

As we already pointed out in \cite[Subsection 1.1]{Ch2}, if $\Omega$ is of class $C^4$ and $\Gamma_0$ is an arbitrary non-empty open subset of $\Gamma$ then we find $\psi \in C^4(\overline{\Omega})$ satisfying $\mathfrak{m}_\psi>0$ such that $\Gamma^\psi\subset \Gamma_0$.

In the rest of this text, $\mathbf{c}>0$ will denote a generic constant depending only on $\Omega$, $A$, $\mathfrak{t}$ and $\psi$. For simplicity, the non-decreasing sequence consisting of the eigenvalues of the operator $-\Delta_{\mathbf{g}}$ with domain $H_0^1(\Omega)\cap H^2(\Omega)$ will denoted by $(\lambda_j)_{j\ge 1}$.

We can now state the results that we want to establish, where $\Sigma^\psi=\Sigma_{\mathfrak{t}}^\psi$.

\begin{theorem}\label{MT1}
Assume that $\Omega$ is $C^{1,1}$ or $\Omega$ is convex. Let $0\le \psi \in C^4(\overline{\Omega})$ such that $\mathfrak{m}_\psi>0$. For any $\lambda \ge \lambda_1$ and $u\in \mathcal{H}_0(Q)$ satisfying $Hu=0$ we have
\begin{align}
&\|u\|_{L^2((0,\mathfrak{t}),H^1(\Omega))}+\|u\|_{L^\infty((0,\mathfrak{t}),L^2(\Omega))}\label{e1.1}
\\
&\hskip 4cm \le \mathbf{c}e^{\lambda \mathfrak{t}}\|\partial_\nu u\|_{L^2(\Sigma^\psi)}+\lambda^{-1/2}\|u(\cdot, 0)\|_{H_0^1(\Omega)}.\nonumber
\end{align}
Furthermore, if  $u(\cdot ,0)\in H^\theta(\Omega)$, for some $\theta >1$, then
\begin{align}
&\|\partial_tu\|_{L^2(Q)}+\|\Delta u\|_{L^2(Q)} \label{e1.2}
\\
&\hskip .5cm \le \mathbf{c}\mathbf{e}_\theta\|u(\cdot,0)\|_{H^\theta(\Omega)}^{1/\theta}\left( e^{\lambda \mathfrak{t}} \|\partial_\nu u\|_{L^2(\Sigma^\psi)}+\lambda^{-1/2}\|u(\cdot, 0)\|_{H_0^1(\Omega)}\right)^{1-1/\theta}.\nonumber
\end{align}
\end{theorem}

\begin{remark}\label{R1}
{\rm
(1) When $\Omega$ is smooth compact connected Riemannian manifold with boundary $\Gamma$, $\Delta_{\mathbf{g}}$ is the Laplace-Beltrami operator on $\Omega$ and $\Upsilon =\omega\times (0,\mathfrak{t})$ (resp. $\Upsilon =\Gamma_0\times (0,\mathfrak{t})$, where $\omega$ (resp. $\Gamma_0$) is a nonempty open subset of $\Omega$ (resp. $\Gamma$), Theorem \ref{MT1} still valid if we replace $\|\partial_\nu u\|_{L^2(\Sigma^\psi)}$ by $\|\partial_\nu u\|_{L^2(\Upsilon)}$. This follows by using \cite[Corollaire 2 and Corollaire 4]{LR} instead of Theorem \ref{O1}.
\\
(2) Let $F\subset Q$ be a measurable set of positive measure. The following observability inequality was established in \cite[Theorem 2]{BM}
\[
\|u(\cdot ,\mathfrak{t})\|_{L^2(\Omega)}\le \mathbf{c}\|\partial_\nu u\|_{L^1(F)},\quad u\in \mathcal{H}_0(Q),\; Hu=0,
\]
where the constant $\mathbf{c}$ depends also on $F$. Again, Theorem \ref{MT1} remains valid whenever we substitute $\|\partial_\nu u\|_{L^2(\Sigma^\psi)}$ by $\|\partial_\nu u\|_{L^1(F)}$ in  \eqref{e1.1} and \eqref{e1.2}.
\\
(3) In the case where $q\in L^\infty (\Omega)$, $\mathbf{g}=\mathbf{I}$ and $\Upsilon=\omega\times (0,\mathfrak{t})$, where $\omega$ is a nonempty open subset of $\Omega$, it was proved in \cite{FZ} the following observability inequality
\[
\|u(\cdot ,\mathfrak{t})\|_{L^2(\Omega)}\le \mathbf{c}\|u\|_{L^2(\Upsilon)},\quad u\in \mathcal{H}_0(Q),\; (H+q)u=0.
\]
Let $\delta =\sup q$ if $\sup q>0$ and $\delta=0$ if $q\le 0$. In light of Lemma \ref{L1} in Appendix \ref{appendixA}, we can proceed as in the proof of Theorem \ref{MT1} in order to obtain the following result: there exist $\lambda_\ast$ depending only on $\Omega$, $\mathbf{g}$ and $q$ so that for all $u\in \mathcal{H}(Q)$ satisfying $(H+q)u=0$ and $\lambda \ge \lambda_\ast$ we have
\begin{align*}
&\|u\|_{L^2((0,\mathfrak{t}),H^1(\Omega))}+\|u\|_{L^\infty((0,\mathfrak{t}),L^2(\Omega))}
\\
&\hskip 4cm\le \mathbf{c}e^{(\delta+\lambda )\mathfrak{t}}\|u\|_{L^2(\Upsilon)}+\lambda^{-1/2}\|u(\cdot, 0)\|_{H_0^1(\Omega)}.
\end{align*}
Additionally, if $u(\cdot ,0)\in H^\theta(\Omega)$, for some $\theta >1$, then
\begin{align*}
&\|\partial_tu\|_{L^2(Q)}+\|\Delta u\|_{L^2(Q)} 
\\
&\hskip .5cm \le \mathbf{c}\mathbf{e}_\theta\|u(\cdot,0)\|_{H^\theta(\Omega)}^{1/\theta}\left( e^{(\delta+\lambda )\mathfrak{t}}\|u\|_{L^2(\Upsilon)}+\lambda^{-1/2}\|u(\cdot, 0)\|_{H_0^1(\Omega)}\right)^{1-1/\theta}.
\end{align*}
In the above two inequalities the constant $\mathbf{c}$ depends also on $q$.
}
\end{remark}

\begin{theorem}\label{MT2}
Assume that $0\le \psi \in C^4(\overline{\Omega})$ is $\mathbf{g}$-pseudo-convex and $\mathfrak{m}_\psi>0$. Then for any $u\in \mathcal{H}_0(Q)$ satisfying $Su=0$ we have
\begin{equation}\label{e2}
\|u\|_{L^\infty ((0,\mathfrak{t}),H^1(\Omega ))}\le \mathbf{c} \|\partial_\nu u\|_{L^2(\Sigma^\psi)}.
\end{equation}
\end{theorem}

\begin{theorem}\label{MT3}
Suppose that $0\le \psi \in C^4(\overline{\Omega})$ is $\mathbf{g}$-pseudo-convex and $\mathfrak{m}_\psi>0$. Let $p,q\in L^\infty(\Omega)$ be nonnegative. There exists $\mathfrak{t}_0$ depending only on $\Omega$, $\mathbf{g}$, $\mathfrak{t}$, $p$, $q$ and $\psi$ so that  for any $\mathfrak{t}\ge \mathfrak{t}_0$ and $u\in \mathcal{H}_1(Q)$ satisfying $(W+p\partial_t+q)u=0$ we have
\begin{equation}\label{e3}
\|u\|_{L^\infty ((0,\mathfrak{t}),H^1(\Omega ))}\le \mathbf{c} \|\partial_\nu u\|_{L^2(\Sigma^\psi)}.
\end{equation}
Here the constant $\mathbf{c}$ depends also on $p$ and $q$.
\end{theorem}

\begin{remark}\label{R2}
{\rm
Let $\Gamma_0$ a nonempty open subset of $\Gamma$ and $\Upsilon =\Gamma_0\times (0,\mathfrak{t})$. 
\\
(1) Assume that $\Omega $, $g_{k\ell}$, $1\le k,\ell$, $p$ and $q$ are $C^\infty$. Under the assumption that $\Upsilon$ satisfies the geometric assumptions of \cite[Theorem 3.3]{BLR} or \cite[Theorem 3.4]{BLR}, it follows from \cite[Theorem 3.8]{BLR} that for each $u\in \mathcal{H}_1(Q)$ satisfying $(W+p\partial_t+q)u=0$ we have
\[
\|u\|_{H^1(Q)}\le \mathbf{c} \|\partial_\nu u\|_{L^2(\Upsilon)},
\]
where the constant $\mathbf{c}$ depends also on $p$ and $q$.
\\
(2) (Transmutation method) In the case $\mathbf{g}=\mathbf{I}$, it is shown in \cite[Section 7.5]{TW} that if there exists $\mathfrak{t}>0$ so that the following observability inequality holds
\[
\|u(\cdot ,0)\|_{H_0^1(\Omega)}+\|\partial_tu(\cdot,0)\|_{L^2(\Omega)}\le \mathbf{c}\|\partial_\nu u\|_{L^2(\Upsilon)},\quad u\in \mathcal{H}_1(Q),\; Wu=0
\]
then for all $\mathfrak{t}>0$ the following observability inequality is satisfied
\[
\|u(\cdot ,0)\|_{H_0^1(\Omega)}\le \mathbf{c}\|\partial_\nu u\|_{L^2(\Upsilon)},\quad u\in \mathcal{H}_0(Q),\; Su=0.
\]
(3) In the case $\mathbf{g}=\mathbf{I}$ and $\psi (x)=|x-x_0|^2$, for some fixed $x_0\in \mathbb{R}^n \setminus\overline{\Omega}$, as we already mentioned $\psi$ is $\mathbf{I}$-pseudo-convex. Since $\Gamma^\psi =\{x\in \Gamma;\; (x-x_0|\nu (x))>0\}$, Theorem \ref{TO3} contains the classical observability inequality obtained by the multiplier method (but with less precise constants) (e.g \cite[Theorem 7.2.4]{TW}).
\\
(4) Let $q\in W^{2,\infty}(\Omega)$, $\Gamma_0$ be an arbitrary nonempty open subset of $\Gamma$, $\mathfrak{t}>0$ and $\Upsilon=\Gamma_0\times (0,\mathfrak{t})$. As a particular case of \cite[Theorem 6.5]{LL} we have
\begin{equation}\label{O4}
\|u(\cdot ,0)\|_{L^2(\Omega)}\le e^{\mathfrak{b}\lambda}\|\partial_\nu u\|_{L^2(\Upsilon)}+\lambda^{-1}\|u(\cdot,0)\|_{H^2(\Omega)},\quad \lambda \ge \lambda_0,
\end{equation}
for all $u\in \mathcal{H}_0(Q)$ satisfying $u(\cdot ,0)\in H^2(\Omega)$ and $(S+q)u=0$, where the constants $\lambda_0$ and $\mathfrak{b}$ only depend on $\Omega$, $A$, $\Gamma_0$ and $\mathfrak{t}$.

We can combine \eqref{e9} in Appendix \ref{appendixA} and \eqref{O4} in order to get in the case $q\le 0$ the following quantitative uniqueness of continuation inequality
\[
\|u\|_{L^\infty( (0,\mathfrak{t}),L^2(\Omega))}\le e^{\mathfrak{b}\lambda}\|\partial_\nu u\|_{L^2(\Upsilon)}+\lambda^{-1}\|u(\cdot,0)\|_{H^2(\Omega)},\quad \lambda \ge \lambda_0,
\]
for all $u\in \mathcal{H}_0(Q)$ satisfying $u(\cdot ,0)\in H^2(\Omega)$ and $(S+q)u=0$, where $\lambda_0$ and $\mathfrak{b}$ are as above.

A similar approximate observability inequality to \eqref{O4} holds for the operator $W+p\partial_t+q$ (\cite[Theorem 6.1]{LL}).
}
\end{remark}

The results we just stated will be proved in the next section. The energy estimates we used are more or less known. For completeness, we added an appendix devoted to the proof of these energy  estimates.

For clarity, we limited our results to some standard evolution equations. However many other results can be derived from the huge literature on observability inequalities, including the magnetic Schr\"odinger and wave equations.

\section{Proof of Theorems \ref{MT1}, \ref{MT2} and \ref{MT3}}

 The proof of Theorem \ref{MT1} is based of the following final-time observability inequality.
 
 \begin{theorem}\label{TO1}
 {\rm (}\cite[Theorem 5.4]{Ch} and \cite{Ch2}{\rm )} Suppose that $\Omega$ is of class $C^1$ or $\Omega$ is convex. Let $\psi \in C^4(\overline{\Omega})$ such that $\mathfrak{m}_\psi>0$. Then for each $u\in \mathcal{H}_0(Q)$ satisfying $Hu=0$ we have
 \begin{equation}\label{O1}
 \|u(\cdot ,\mathfrak{t})\|_{H_0^1(\Omega)}\le \mathbf{c} \|\partial_\nu u\|_{L^2(\Sigma^\psi)}.
 \end{equation}
 \end{theorem}

\begin{proof}[Proof of Theorem \ref{MT1}]
Let $\phi_j$, $j\ge 1$, be an eigenfunction for $\lambda_j$ so that $(\phi_j)_{j\ge 1}$ forms an orthonormal basis of $L^2(\Omega)$. In this case $u(\cdot,\mathfrak{t})$ is given by the series
\[
u(\cdot ,\mathfrak{t})=\sum_{j\ge 1}e^{-\lambda_j \mathfrak{t}}(u(\cdot ,0),\phi_j)\phi_j.
\]
Here and henceforth, $(\cdot ,\cdot)$ denotes the usual scalar product of $L^2(\Omega)$. In particular, we have
\begin{equation}\label{e.15}
(u(\cdot ,0),\phi_j)=e^{\lambda_j \mathfrak{t}}(u(\cdot ,\mathfrak{t}),\phi_j),\quad j\ge 1.
\end{equation}

Let $\lambda \ge \lambda_1$. In light of \eqref{e.15}, we get
\begin{equation}\label{e.16}
\sum_{\lambda_j\le \lambda}|(u(\cdot, 0),\phi_j)|^2\le e^{2\lambda \mathfrak{t}}\sum_{\lambda_j\le \lambda}|(u(\cdot, \mathfrak{t}),\phi_j)|^2\le e^{2\lambda \mathfrak{t}}\|u(\cdot, \mathfrak{t})\|_{L^2(\Omega)}^2.
\end{equation}

On the other hand, since $u(\cdot, 0)\in H_0^1(\Omega)$, we obtain
\begin{equation}\label{e.17}
\sum_{\lambda_j> \lambda}|(u(\cdot, 0),\phi_j)|^2\le \lambda^{-1}\sum_{\lambda_j> \lambda}\lambda_j|(u(\cdot, 0),\phi_j)|^2\le \lambda^{-1}\|u(\cdot, 0)\|_{H_0^1(\Omega)}^2.
\end{equation}

Putting together \eqref{e.16} and \eqref{e.17}, we find
\begin{equation}\label{e.18}
\|u(\cdot,0)\|_{L^2(\Omega)}\le e^{\lambda \mathfrak{t}}\|u(\cdot, \mathfrak{t})\|_{L^2(\Omega)}+\lambda^{-1/2}\|u(\cdot, 0)\|_{H_0^1(\Omega)}.
\end{equation}
Then \eqref{O1} in \eqref{e.18} yields
\begin{equation}\label{e.19}
\|u(\cdot,0)\|_{L^2(\Omega)}\le \mathbf{c}e^{\lambda \mathfrak{t}}\|\partial_\nu u\|_{L^2(\Sigma^\psi)}+\lambda^{-1/2}\|u(\cdot, 0)\|_{H_0^1(\Omega)}.
\end{equation}
This inequality, \eqref{e4} and \eqref{e5} in Appendix \ref{appendixA} with $q=0$ and $\delta=0$ give 
\[
\|u\|_{L^2((0,\mathfrak{t}),H^1(\Omega))}+\|u\|_{L^\infty((0,\mathfrak{t}),L^2(\Omega))}\le \mathbf{c}e^{\lambda \mathfrak{t}}\|\partial_\nu u\|_{L^2(\Sigma^\psi)}+\lambda^{-1/2}\|u(\cdot, 0)\|_{H_0^1(\Omega)}.
\]
That is \eqref{e1.1} is proved.

Next, using  the fact that $\|\partial_tu\|_{L^2(Q)}=\|\Delta u\|_{L^2(Q)}$, we obtain from \eqref{e4}, \eqref{e7} in Appendix \ref{appendixA}, again  with $q=0$ and $\delta =0$, and Poincar\'e's inequality
\[
\|\partial_tu\|_{L^2(Q)}+\|\Delta u\|_{L^2(Q)}\le \mathbf{c}\|u(\cdot ,0)\|_{H_0^1(\Omega)}.
\]
We then use the interpolation inequality \eqref{i} in order to obtain
\[
\|\partial_tu\|_{L^2(Q)}+\|\Delta u\|_{L^2(Q)}\le \mathbf{c} \mathbf{e}_\theta \|u(\cdot ,0)\|_{H^\theta(\Omega)}^{1/\theta}\|u(\cdot ,0)\|_{L^2(\Omega)}^{1-1/\theta}.
\]
This inequality and \eqref{e1.1}  give \eqref{e1.2}.
\end{proof}

Before proving Theorem \ref{MT2} we recall the following observability inequality.

\begin{theorem}\label{TO2}
{\rm (}\cite[Theorem 6.6]{Ch}{\rm )} Let $0\le \psi \in C^4(\overline{\Omega})$ be $\mathbf{g}$-pseudo-convex and satisfies $\mathfrak{m}_\psi>0$. Then for each $u\in \mathcal{H}_0(Q)$ satisfying $Su=0$ we have
\begin{equation}\label{O2}
\|u(\cdot ,0)\|_{H_0^1(\Omega)}\le \mathbf{c}\|\partial_\nu u\|_{L^2(\Sigma ^\psi)}.
\end{equation}
\end{theorem}

\begin{proof}[Proof of Theorem \ref{MT2}]
In light of \eqref{e9} and \eqref{e8} in Appendix \ref{appendixA} with $q=0$ we obtain
\[
\|u\|_{L^\infty((0,\mathfrak{t}),H^1(\Omega))}\le \mathbf{c}\|u(\cdot ,0)\|_{H_0^1(\Omega)}.
\]
Combining this inequality and \eqref{O2}, we find \eqref{e3}.
\end{proof}

Theorem \ref{MT3} is an immediate consequence of Lemma \ref{L3} in Appendix \ref{appendixA} and the following observability inequality
\begin{theorem}\label{TO3}
{\rm (}\cite[Theorem 3.10]{Ch}{\rm )} Assume that $0\le \psi \in C^4(\overline{\Omega})$ is $\mathbf{g}$-pseudo-convex and $\mathfrak{m}_\psi>0$. Let $p,q\in L^\infty(\Omega)$. There exists $\mathfrak{t}_0>0$ depending only on $\Omega$, $\mathbf{g}$,  $p$, $q$ and $\psi$ so that  for any $u\in \mathcal{H}_1(Q)$ satisfying $(W+p\partial_t+q)u=0$ we have
\begin{equation}\label{O3}
\|u(\cdot ,0)\|_{H_0^1(\Omega )}\le \mathfrak{c} \|\partial_\nu u\|_{L^2(\Sigma^\psi)}.
\end{equation}
Here the constant $\mathbf{c}$ depends also on $p$ and $q$.
\end{theorem}

\appendix
\section{Energy estimates}\label{appendixA}

\begin{lemma}\label{L1}
Let $q\in L^\infty (\Omega )$ real valued, $\delta =\sup q$ if $\sup q>0$ and $\delta=0$ if $q\le 0$. Let $u\in \mathcal{H}_0(Q)$ such that $(H+q)u=0$. Then
\begin{equation}\label{e4}
\sqrt{2}\|\sqrt{\mathbf{g}}\, \nabla u\|_{L^2(Q)}\le e^{\delta \mathfrak{t}}\|u(\cdot ,0)\|_{L^2(\Omega)},
\end{equation}
\begin{equation}\label{e5}
\|u(\cdot ,t)\|_{L^2(\Omega)}\le e^{\delta \mathfrak{t}}\|u(\cdot ,0)\|_{L^2(\Omega)},\quad t\in [0,\mathfrak{t}]
\end{equation}
and
\begin{align}
&\|\partial_tu\|_{L^2(Q)}\le  e^{\delta \mathfrak{t}}\left(\|\sqrt{\mathbf{g}}\, \nabla u(\cdot ,0)\|_{L^2(\Omega)}+\sup \sqrt{\delta -q}\,  \|u(\cdot ,0)\|_{L^2(\Omega)}\right) \label{e7}
\\
&\hskip 9cm+\delta \|u\|_{L^2(\Omega)}.\nonumber
\end{align}
\end{lemma}

\begin{proof}
First, note that $v=e^{-\delta t}u$ satisfies
\begin{equation}\label{e6}
\Delta_{\mathbf{g}} v-\partial_tv-(\delta -q)v=0.
\end{equation}
Multiplying \eqref{e6} by $v$ and making integration by parts in order to obtain
\[
\int_{Q_t}|\sqrt{\mathbf{g}}\, \nabla v|^2 \, dx ds +\int_{Q_t}(\delta -q)v^2\, dx ds=\frac{1}{2}\int_\Omega v^2(\cdot,0)\; dx - \frac{1}{2}\int_\Omega v^2(\cdot,\mathfrak{t})\; dx,
\]
from which we derive \eqref{e4} and \eqref{e5}.

Since
\[
\partial_t \int_\Omega |\sqrt{\mathbf{g}}\,\nabla v(\cdot,t) |^2\, dx =2\int_\Omega \mathbf{g}\nabla v(\cdot,t)\cdot \nabla \partial_tv(\cdot,t)\, dx ,
\]
an integration by parts then gives
\begin{align*}
\partial_t \int_\Omega |\sqrt{\mathbf{g}}\, \nabla v(\cdot,t)|^2\, dx &=-2\int_Q\Delta_{\mathbf{g}} v(\cdot,t)\partial_tv(\cdot,t)\, dx 
\\
&=-2\int_\Omega (\partial_tv(\cdot ,t))^2\, dx-\partial_t\int_\Omega (\delta -q)v^2(\cdot,t)\, dx.
\end{align*}
Integrating with respect to $t$, we obtain
\begin{align*}
&\|\sqrt{\mathbf{g}}\, \nabla v(\cdot ,\mathfrak{t})\|_{L^2(\Omega)}^2-\|\sqrt{\mathbf{g}}\, \nabla v(\cdot ,0)\|_{L^2(\Omega)}^2
\\
&\hskip 2cm  =-2\|\partial_tv\|_{L^2(Q)}^2-\|\sqrt{\delta -q}\, v(\cdot ,\mathfrak{t})\|_{L^2(\Omega)}^2+\|\sqrt{\delta -q}\, v(\cdot ,0)\|_{L^2(\Omega)}^2
\end{align*}
and hence 
\[
\|\partial_tv\|_{L^2(Q)}\le  \|\sqrt{\mathbf{g}}\, \nabla v(\cdot ,0)\|_{L^2(\Omega)}+\sup \sqrt{\delta -q}\, \|v(\cdot ,0)\|_{L^2(\Omega)}.
\]
This inequality together with $\partial_tu=e^{\delta t}\partial_tv+\delta u$ give 
\[
\|\partial_tu\|_{L^2(Q)}\le  e^{\delta \mathfrak{t}}\left(\|\sqrt{\mathbf{g}}\, \nabla u(\cdot ,0)\|_{L^2(\Omega)}+\sup \sqrt{\delta -q}\,  \|u(\cdot ,0)\|_{L^2(\Omega)}\right)+\delta \|u\|_{L^2(Q)}.
\]
That is we proved \eqref{e7}.
\end{proof}

\begin{lemma}\label{L2}
Let $q\in L^\infty (\Omega)$ be non-positive. For all $u\in \mathcal{H}_0(Q)$ satisfying $(S+q)u=0$ and $t\in (0,\mathfrak{t}]$ we have
\begin{equation}\label{e9}
\|u(\cdot ,t)\|_{L^2(\Omega)}=\|u(\cdot ,0)\|_{L^2(\Omega)},
\end{equation}
and
\begin{align}
&\|\sqrt{\mathbf{g}}\, \nabla u(\cdot ,t)\|_{L^2(\Omega)}^2+\|\sqrt{-q}\, u(\cdot ,t)\|_{L^2(M)}^2 \label{e8}
\\
&\hskip 4cm =\|\sqrt{\mathbf{g}}\, \nabla u(\cdot ,0)\|_{L^2(\Omega)}^2+\|\sqrt{-q}\, u(\cdot ,0)\|_{L^2(\Omega)}^2.\nonumber
\end{align}
\end{lemma}

\begin{proof}
Since
\begin{align*}
\partial_t |u|^2&=\partial_tu \overline{u}+u\partial_t \overline{u}
\\
&= i(\Delta_{\mathbf{g}} u+qu)\overline{u}-iu(\Delta_{\mathbf{g}} \overline{u}+q\, \overline{u})
\\
&= i\Delta_{\mathbf{g}} u\overline{u}-iu\Delta_{\mathbf{g}} \overline{u},
\end{align*}
a simple integration by parts give
\[
\int_\Omega (\Delta_{\mathbf{g}} u(\cdot,s)\overline{u}(\cdot,s)-u(\cdot,s)\Delta_{\mathbf{g}} \overline{u}(\cdot,s))\, dx=0,\quad s\in (0,t).
\]
Hence
\[
\int_{Q_t}\partial_t |u|^2\, dx ds =0
\]
and then
\[
\|u(\cdot ,t)\|_{L^2(\Omega)}=\|u(\cdot ,0)\|_{L^2(\Omega)}.
\]
That is we proved \eqref{e9}

On the other hand, we have
\begin{align*}
\partial_t\int_\Omega |\sqrt{\mathbf{g}}\, \nabla u|^2\, dx &=\int_\Omega \mathbf{g}\nabla \partial_tu\cdot \nabla \overline{u}\, dx +\int_\Omega \mathbf{g}\nabla u\cdot \nabla \partial_t\overline{u}\, dx 
\\
&=-\int_\Omega \partial_tu\Delta_{\mathbf{g}} \overline{u}\, dx -\int_\Omega \Delta_{\mathbf{g}} u \partial_t\overline{u}\, dx
\\
&=\int_\Omega \partial_tu(-i\partial_t\overline{u}+q\overline{u})\, dx +\int_\Omega (i\partial_t u+qu) \partial_t\overline{u}\, dx 
\\
&=\partial_t\int_\Omega q|u|^2\, dx,
\end{align*}
from which we get 
\begin{equation}\label{e.10}
\int_\Omega (|\sqrt{\mathbf{g}}\, \nabla u(\cdot,\mathfrak{t})|^2-q|u(\cdot,\mathfrak{t})|^2)\, dx=\int_\Omega(|\sqrt{\mathbf{g}}\,\nabla u(\cdot,0)|^2-q|u(\cdot,0)|^2)\, dx.
\end{equation}
We rewrite this identity  in the form
\begin{align*}
\|\sqrt{\mathbf{g}}\, \nabla u(\cdot ,\mathfrak{t})\|_{L^2(\Omega)}^2+\|\sqrt{-q}\, u(\cdot ,\mathfrak{t})&\|_{L^2(\Omega)}^2
\\
&=\|\sqrt{\mathbf{g}}\, \nabla u(\cdot ,0)\|_{L^2(\Omega)}^2+\|\sqrt{-q}\, u(\cdot ,0)\|_{L^2(\Omega)}^2.
\end{align*}
In other words, we proved \eqref{e8}.
\end{proof}

\begin{lemma}\label{L3}
Let $p,q\in L^\infty(\Omega)$ be non-positive. For all $u\in \mathcal{H}_1(Q)$ satisfying $(W+p\partial_t+q)u=0$ and $t\in (0,\mathfrak{t}]$ we have
\begin{align}
&\|\partial_tu(\cdot ,t)\|_{L^2(\Omega)}^2+\|\sqrt{\mathbf{g}}\, \nabla u(\cdot ,t)\|_{L^2(\Omega)}^2+\|\sqrt{-q}\, u(\cdot ,t)\|_{L^2(\Omega)}^2\label{e.11}
\\
&\hskip 1.5cm \le \|\partial_tu(\cdot ,0)\|_{L^2(\Omega)}^2+\|\sqrt{\mathbf{g}}\, \nabla u(\cdot ,0)\|_{L^2(\Omega)}^2+\|\sqrt{-q}u(\cdot ,0)\|_{L^2(\Omega)}^2.\nonumber
\end{align}
\end{lemma}

\begin{proof}
We have
\begin{align}
&\int_{Q_t}\partial_t^2u\partial_t\overline{u}\, dx ds -\int_{Q_t}\Delta_{\mathbf{g}} u\partial_t\overline{u}\, dx ds\label{e.12}
\\
&\hskip 4cm-\int_qp\partial_tu\partial_t\overline{u}\, dx ds-\int_{Q_t}qu\partial_t\overline{u}\, dx ds=0,\nonumber
\\
&\int_{Q_t}\partial_t^2\overline{u}\partial_tu\, dx ds -\int_{Q_t}\Delta_{\mathbf{g}} \overline{u}\partial_tu\, dx ds\label{e.13}
\\
&\hskip 4cm -\int_qp\partial_t\overline{u}\partial_tu\, dx ds-\int_{Q_t}q\overline{u}\partial_tu\, dx ds=0.\nonumber
\end{align}
Applying Green's formula we obtain
\begin{align}
-\int_{Q_t}\Delta_{\mathbf{g}} u\partial_t\overline{u}\, dxds-\int_{Q_t}&\Delta_{\mathbf{g}} \overline{u}\partial_tu\, dxds\label{e.14}
\\
&=\int_{Q_t}\mathbf{g}\nabla u\cdot \nabla \partial_t\overline{u}\, dxds+\int_{Q_t}\mathbf{g}\nabla \overline{u}\cdot\nabla \partial_tu\, dx ds\nonumber
\\
&=\int_{Q_t} \partial_t|\sqrt{\mathbf{g}}\, \nabla u|^2\, dxds .\nonumber
\end{align}
In light of \eqref{e.14}, taking the sum side by side of \eqref{e.12} and \eqref{e.13}, we get 
\[
\int_{Q_t}\partial_t(|\partial_tu|^2+|\sqrt{\mathbf{g}}\, \nabla u|^2-q|u|^2)\, dx ds=2\int_{Q_t}p|\partial_tu|^2\le 0.
\]
In consequence, we have
\begin{align*}
&\|\partial_tu(\cdot ,t)\|_{L^2(\Omega)}^2+\|\sqrt{\mathbf{g}}\, \nabla u(\cdot ,t)\|_{L^2(\Omega)}^2+\|\sqrt{-q}\, u(\cdot ,t)\|_{L^2(\Omega)}^2
\\
&\hskip 2cm \le \|\partial_tu(\cdot ,0)\|_{L^2(\Omega)}^2+\|\sqrt{\mathbf{g}}\, \nabla u(\cdot ,0)\|_{L^2(\Omega)}^2+\|\sqrt{-q}\, u(\cdot ,0)\|_{L^2(\Omega)}^2.
\end{align*}
That is we proved \eqref{e.11}
\end{proof}

\subsection*{Conflict of Interest statement} The author declares that there is no conflict of interest regarding the content of this work.

\subsection*{Data availability statement} There is no data associated with this work.

\end{document}